\documentclass[oneside,a4paper,11pt]{amsart}	%Style du document
\usepackage[english]{babel}
\usepackage[T1]{fontenc}
\usepackage{mathrsfs}
\usepackage{amsthm}
\usepackage{amssymb}
\usepackage{amscd}
\usepackage{setspace}
\usepackage{geometry}							%Ajuste les dimensions du document
\usepackage[toc,page]{appendix}
\theoremstyle{plain}							%Ecriture forte, style Théorème
\newtheorem{thm}{\textbf{Theorem}}[section]
\newtheorem{prop}[thm]{Proposition}
\newtheorem{cor}[thm]{Corollary}

\newtheorem{lem}[thm]{Lemma}

\theoremstyle{definition}						%Ecriture style definition

\newtheorem{defi}[thm]{Definition}
\newtheorem{rem}[thm]{Remark}
\newtheorem{ex}[thm]{Example}
\newtheorem{fact}[thm]{Fact}

\theoremstyle{remark}							%Ecriture style remarque
\newcommand{\mo}{\models}						%Modèle de
\newcommand{\mop}{\models^+}					%Modèles positif de
\newcommand{\moe}{\models^e}					%Modèle existentiel de
\newcommand{\fo}{\Vdash}						%Force (Robinson)
\newcommand{\imp}{\Rightarrow}					%Implication
\renewcommand{\iff}{\Leftrightarrow}			%Equivalence (connecteur)
\renewcommand{\and}{\wedge}						%Connecter ET logique
\newcommand{\oo}{\infty}						%Infini

\newcommand{\QQ}{\mathbb Q}						%Ensemble Q des nombres rationnels
\newcommand{\RR}{\mathbb R}						%Ensemble R des nombres réels
\newcommand{\C}{\mc C}							%C
\newcommand{\LS}{\mathcal L}					%Catégorie des L-structures
\newcommand{\M}{\mc M}							%Catégorie de modèles
\newcommand{\xx}{\times}

\renewcommand{\phi}{\varphi}					%Phi plus joli
\newcommand{\ms}{\mathscr}						%Calligraphique fin lié
\newcommand{\mf}{\mathfrak}						%Gothique
\newcommand{\mc}{\mathcal}						%Calligraphique gras standard
\newcommand{\ov}{\overline}						%Surligné
\title{Positive model theory and infinitary logic}
\author{Jean Berthet}

\begin{document}

\begin{abstract}
We study the basic properties of a dual "spectral" topology on positive type spaces of h-inductive theories and its essential connection to infinitary logic. The topology is Hausdorff, has the Baire property, and its compactness characterises positive model completeness; it also has a basis of clopen sets and is described by the formulas of geometric logic. The "geometric types" are closed in the type spaces under all the operations of infinitary logic, and we introduce a positive analogue of existentially universal structures, through which we interpret the full first order logic in positive type spaces. This shows how "positive $\omega$-saturation" is a fundamental connection between positive and infinitary logic, and we suggest a geometric analogue of positive Morleyisation.
\end{abstract}

%Construite d'après la version 2014-12-25 nettoyée.
%Changement éventuel du titre et de l'abstract, pas encore de soumission à arxiv.

\maketitle

\section*{Introduction and background}
Positive model theory was introduced in \cite{PMTCAT} and revisited in \cite{BYP} in terms of the study of positively existentially closed models of an h-inductive theory. In this context, which may be construed as a generalisation of classical first order model theory by positive Morleyisation, the spaces of types still play an essential role. Contrary to the classical case, the definable topology on these spaces, though compact, is not Hausdorff, and the fact that basic definably closed sets are not open in general implies that some properties which are topologically linked to the definable topology in the classical setting, like the existence of coheirs, fail to be reproducible in the positive context.
In this note we introduce and study, as a tool for positive model theory, a topology which is in some sense "dual" to the definable topology, and intrinsically linked to the positively existentially closed models of an h-inductive theory; we call it the "spectral" topology.\\
In section \ref{SPECTOP}, we define the spectral topology and study its basic properties : it is Hausdorff and finer than the definable topology, though not compact in general. Its compactness characterises the positive model completeness of the Kaiser hull of the underlying theory and in general, we may only count on an "infinitary compactness". We also show that the spectral topology has the Baire property, and exhibit a particular basis of clopen sets, the constructible subsets, which are analogous to the basic clopen sets of the classical type spaces, and to which the notion of resultant may be extended. %\\
In section \ref{GEOM}, we introduce a description of spectrally open subsets by geometric formulas, using bounds from the size of the language, and introduce the geometric types, generalising the partial positive types. It turns out that the maximal ones are essential the positive types, whereas in general their power of expression is much greater, as every subset of a positive type space is decribed by such a type. %\\
In section \ref{EXINF}, we connect these infinitary properties to full first order logic, using a positive generalisation of existentially universal structures. The existential models we obtain for h-inductive theories generalise the $\omega$-saturated positively existentially closed ones of \cite{BYP}, and using Karp's theorem on infinitary equivalence we show that every formula of $L_{\oo\omega}$ has an interpretation as a subset of a positive type space. This may be rephrased via positive Morleyisation as a bound on the complexity of formulas necessary to describe the infinitary behavior of finite tuples in $\omega$-saturated structures in the classical setting. %\\
In the same spirit, we finish in section \ref{GEOMOR} by a reduction of infinitary logic to geometric logic by an infinitary analogue of positive Morleyisation.

\subsubsection*{Conventions and notations}
Our general reference for model theory is \cite{WH}. We work in a many-sorted first order language $L$ (as for instance in \cite{FOCL}, 2.1), the subset of sort symbols of which we note $S$. Relation symbols have a \emph{sorting} and function symbols have an \emph{arity}, which are both a \emph{finite} string (or "tuples") of sort symbols; in addition function (and constant) symbols have a \emph{sort}, which is a sort symbol; languages are interpreted in the classical way. We fix the presence of two $\emptyset$-ary relation symbols $\top$ (true) and $\bot$ (false), interpreted in the obvious way in each $L$-structure. We work with a set $V$ of variables, each coming with a sort, and each sort having a countable supply of variables; we may take $V=S\times \omega$, with a variable $(s,i)\in V$ having sort $s$.\\
Formulas are considered only in $L_{\infty\omega}$ and are noted $\phi(x)$; striclty speaking this denotes a \emph{couple} $(\phi,x)$, where $\phi$ is a formula which free variables are among the finite tuple $x$ of variables. We use the same convention for terms $t(x)$. We note $V^*$ the set of all finite tuples of variables. Such tuples are noted by single letters $x,y,z,\ldots$; the expression $x\cap y=\emptyset$ is intended to mean that $x$ and $y$ have no common variables, while the notation $|x|=|y|$ means that $x$ and $y$ have the same \emph{sorting}, i.e. the sorts of the variables appearing in order in $x$ and $y$ are the same. A formula is \emph{positive (existential)} if is is finitary and mentions only finite conjunctions, disjunctions and existential quantifications; we note $L^+$ the set of positive formulas.\\
A \emph{sorted map of $L$-structures} $f:A\to B$ is a family $(f_s)_{s\in S}$ of maps $f_s:A_s\to B_s$ for each sort symbol $s\in S$; it is a \emph{($L$-)homomorphism} if for every atomic sentence $\phi(a)$ with parameters in $A$ such that $A\mo \phi(a)$, we have $B\mo \phi(fa)$. A homomorphism $f:A\to B$ is an \emph{immersion} if for every positive sentence $\phi(a)$ with parameters in $A$, if $B\mo \phi(fa)$ then $A\mo \phi(a)$.\\
If $\kappa$ is a %regular
cardinal, by $L_{\kappa\omega}$ we mean the smallest subclass of formulas of $L_{\oo\omega}$, which contains atomic formulas and is closed under negation, conjuncts or disjuncts of sets of formulas of size $<\kappa$, and finite quantifications. The language $L_{\oo\omega}$ is the union of the classes $L_{\kappa\omega}$. % for $\kappa$ regular. 
We distinguish as in \cite{FOCL} the \emph{geometric formulas} of $L_{\oo\omega}$, which class $L^g_{\oo\omega}$ is the smallest subclass of $L_{\oo\omega}$ containing the atomic formulas and closed under finite conjonctions, existential quantifications over single variables and arbitrary disjunctions, and we note $L^g_{\kappa\omega}=L^g_{\oo\omega}\cap L_{\kappa\omega}$.
Our references for positive model theory are \cite{PMTCAT} and \cite{BYP}, but our exposition follows more closely \cite{BYP}. We review the basic elements in a many-sorted context and provide some special notation and terminology.
\subsubsection*{h-Inductive theories and positive models}
We say with \cite{BYP} that a finitary first order $L$-sentence is \emph{h-inductive}, if it is a finite conjunct of \emph{basic h-inductive sentences}, which have the form $\forall x\ (\phi(x)\imp \psi(x))$, where $\phi(x)$ and $\psi(x)$ are positive formulas.
A first order $L$-theory $T$ is \emph{h-inductive}, if it consists of h-inductive sentences. The models of such a theory $T$ form an \emph{inductive class}, i.e. closed under directed colimits of $L$-homomorphisms. If $\C$ is a class of $L$-structures and $A\in \C$, $A$ is \emph{positively existentially closed in $\C$} if every $L$-homomorphism $f:A\to B$, with $B\in \C$, is an immersion.
\begin{fact}[\cite{BYP}, Theorem 1]\label{INDPEC}
In an inductive class, every structure continues into a positively existentially closed one.
\end{fact}
\noindent If $T$ is h-inductive, we note $\M(T)$ the full subcategory of its models and $\M^+(T)$ the full subcategory of positively existentially closed models of $T$. In order to lighten the terminology and avoid confusion with classical existential completness, we suggest to call the objects of $\M^+(T)$ the \emph{positive models} of $T$. If every model of $T$ is positive (i.e. if $\M(T)=\M^+(T)$) we say that $T$ is \emph{positively model complete}.
\begin{fact}[\cite{BYP}, Lemma 15] 
$T$ is positively model complete if and only if for every positive formula $\phi(x)$, there exists a positive formula $\psi(x)$ such that $T\mo \forall x(\phi(x)\wedge\psi(x)\imp \bot)$ and $T\mo \forall x(\top\imp \phi(x)\vee\psi(x))$.
\end{fact}
\noindent We define as in (\cite{MB}, Definition 9), the \emph{resultant} of a positive formula $\phi(x)$, which is the set $Res_T(\phi(x))$ of all positive formulas $\psi(x)$ such that $T\mo \forall x(\phi(x)\wedge\psi(x)\imp \bot)$. Using Lemma 14 of \cite{BYP}, it is possible to characterise the positive models of $T$ as follows.
\begin{fact}\label{POSCAR}%[\cite{BYP}, Lemma 14]% or\cite{GCR}, Lemma 2.5]
A model $M$ of $T$ is positive if and only if for every positive formula $\phi(x)$, we have $M_x-\phi(x)^M=\bigcup\{\psi(x)^M: \psi\in Res_T(\phi)\}$.  
\end{fact}

%Homomorphismes et diagrammes
\subsubsection*{h-Universal sentences and companions}
%\noindent\textbf{h-Universal sentences and companions}\\
A (basic) h-inductive sentence is \emph{h-universal} if it has the form $\forall x\ (\phi(x)\imp \bot)$, with $\phi(x)$ positive. We note $T_u$ the set of h-universal consequences of $T$.% and recall the
\begin{fact}[\cite{BYP}, Lemma 5]
An $L$-structure $A$ is a model of $T_u$ if and only if there exists an $L$-homomorphism $f:A\to M$ into a model $M$ of $T$.
\end{fact}
\noindent If $A$ is an $L$-structure, we note $L(A)$ the expansion of $L$ by the elements of $A$ naming themselves, and $D^+A$ the atomic diagram of $A$, i.e. the set of atomic sentences with parameters in $A$ which are true in $A$. A model $B$ of $D^+A$ is essentially the same thing as an $L$-homomorphism $f:A\to B$, preserving the canonical interpretation of $A$ in itself.  By the fact, $L$-homomorphisms from $A$ into a model of $T$ are essentially the models of $T_u\cup D^+A$ in the language $L(A)$.\\
%Compagnes
\noindent If $T'$ is another h-inductive theory in the same language, we say that $T$ and $T'$ are \emph{positive companions} if they have the same positive models, i.e. $\M^+(T)=\M^+(T')$. The \emph{Kaiser hull of $T$}, noted $T_k$, is the class of all h-inductives sentences which are satisfied in every positive model of $T$.
\begin{fact}[\cite{BYP}, Lemma 7]
$T_u$ is the smallest positive companion of $T$ and $T_k$ is the largest positive companion of $T$.
\end{fact}

\subsubsection*{Spaces of positive types}
%Espaces de types
\noindent If $x$ is a (possibly infinite) tuple of variables (or of new type constants), a \emph{positive type in $T$ in variables $x$} is a set $p$ of positive formulas $\phi(x)$, such that $T\cup p(x)$ is consistent, and $p(x)$ is maximal with this property. We write $S_x(T)$ the set of all such positive types. The sets of the form $[\phi(x)]=\{p\in S_x(T): \phi\in \phi\}$, for positive formulas $\phi(x)$, are closed under finite unions and intersections, and are the basic closed sets of the \emph{definable topology} on $S_x(T)$, which we will note $\ms D$. If $A$ is an $L$-structure, $x$ is a finite tuple of variables and $a\in M_x$, the \emph{positive type of $a$ in $M$}, noted $tp^+_A(a)$, is the set of all positive formulas $\phi(x)$ such that $A\mo \phi(a)$.

\begin{fact}[\cite{BYP}, Lemma 13]
For every finite tuple of variables $x$, the positive types of $S_x(T)$ are the positive types of corresponding tuples in positive models of $T$, and we have $S_x(T_u)=S_x(T)=S_x(T_k)$.%\\
\end{fact}
\begin{fact}[\cite{BYP}, Lemma 16]
The definable topology is compact, though not Hausdorff in general.
\end{fact} 

\section{The spectral topology}\label{SPECTOP}
From now on, $T$ is an h-inductive theory in a first order language $L$, $x$ a finite tuple of variables, $S_x(T)$ is the set of positive types in the tuple $x$. If $\phi(x)\in L^+$, we recall that $[\phi(x)]=\{p\in S_x(T): \phi\in p\}$.

\begin{defi}
The \emph{spectral topology on $S_x(T)$} is the topology $\ms S$ which basis of \emph{open} sets are the sets $[\phi]$, for $\phi(x)\in L^+$.
\end{defi}

\begin{prop}\label{TOPSPEC}
The spectral topology is Hausdorff and finer than the definable topology.
\end{prop}
\begin{proof}
Suppose that $p\in S_x(T)$ and $\phi(x)\in L^+$. If $p\notin [\phi]$, let $M\mop T$ and $a\in M_x$ a realisation of $p$ in $M$ : as $M$ is positive, there exists $\psi\in Res_T(\phi)$ such that $M\mo \psi(a)$, and as $p$ is maximal, we have $\psi\in p$, i.e. $p\in [\psi]$. Reciprocally, if $\psi\in Res_T(\phi)$ and $p\in [\psi]$, by consistency of $p$ we have $\phi\notin p$, i.e. $p\notin [\phi]$. This means we have $S_x(T)-[\phi]=\bigcup\{[\psi]: \psi\in Res_T(\phi)\}$ and this last is open for the spectral topology, so $[\phi]$ itself is spectrally closed. The basic definably closed sets $[\phi]$ are closed for $\ms S$, hence $\ms S$ is finer than $\ms D$. As for Hausdorff separation, if $q\neq p$ is another positive type, by maximality of such there exists a formula $\phi$ such that $\phi\in p-q$; this means that $p\in [\phi]$, whereas $q\in[\phi]^c$ and we have just seen that this last set is spectrally open, so the disjoint opens $[\phi]$ and $[\phi]^c$ separate $p$ and $q$.
\end{proof}

\begin{ex}
Let $T$ be the h-inductive theory of (strict) linear orders in the language $\{<\}$ : its positive models are the dense linear orders, which is a consequence of quantifier elimination for $T_k$, or may be checked directly. The rational order $(\QQ,<)$ is a positive model, and the assignation to every element of $\ov\RR=\RR\cup\{-\infty,+\infty\}$ of its (positive) type over $\QQ$ is a bijection between $\ov\RR$ %with the "constructible topology" (i.e. the full definable topology, with as basis the intervals) 
and $S_1(\QQ)$. The spectral topology on $S_1(\QQ)$ induces a topology on $\RR$, which happens to be the classical definable one by the following proposition..
\end{ex}

\begin{prop}\label{SPEC}
$T_k$ is positively model complete if and only if for every finite tuple $x\in V^*$ the space $S_x(T)$ is compact for the spectral topology, and in this case the two topologies coincide with the classical definable topology (and $\M^+(T)=\M(T_k)$).
\end{prop}
\begin{proof}
If $T_k$ is positively model complete, every $L$-formula is equivalent modulo $T_k$ to a positive formula, hence the positive types spaces $S_x(T)=S_x(T_k)$ are homeomorphic to the classic type spaces (for the spectral topology), hence the positive type spaces are compact Hausdorff.\\
Reciprocally, suppose that every space of (positive) types is spectrally compact and let $\phi(x)\in L^+$. In $S_x(T)$ the complement $S_x(T)-[\phi]=\bigcup_{\psi\in Res_T(\phi)} [\psi]$ is spectrally closed, hence compact, so we may find a finite subset of $Res_T(\phi)$, in fact a single $\psi\in Res_T(\phi)$, such that $S_x(T)-[\phi]=[\psi]$. Now if $M\mop T$ and $a\in M_x$ is such that $M\not\mo \phi(a)$, we have $tp(a)=p\notin [\phi]$, hence $p\in [\psi]$, which means that $M\mo \psi(a)$. In other words, $\psi(x)$ defines a complement of $\phi$ modulo $T_k$, which is then positively model complete (and axiomatises the positive models of $T$).
\end{proof}

\begin{ex}
(i) There are many examples where $T_k$ is not positively model complete; for instance, if $T_k$ is not Hausdorff (see \cite{PMTCAT} and \cite{BYP}), then the definable topology cannot coincide with the spectral topology, which is then not compact.\\
(ii) A great counterexample to the proposition is given by the h-inductive theory $T$ of division rings in the language $(+,\xx,-,0,1)$ of rings, as mentionned in \cite{FAD}, Chapter 14. Indeed, one checks that the division rings are axiomatisable by h-inductive sentences, and that the positive division rings are exactly the existentially closed (in the classical sense) division rings, because every existential formula is equivalent modulo $T$ to a positive formula, replacing inequations $P(x)\neq 0$ by $(\exists y)\ P(x).y=1$. It is well known that these do not form an axiomatisable class, hence $T_k$ is not positively model complete.
\end{ex}

\begin{defi}
If $\kappa$ is a cardinal, say that a topological space $X$ is \emph{$\kappa$-compact} if for every open cover $X=\bigcup_{i\in I} O_i$, there exists a subset $J\subset I$, such that $|J|<\kappa$ and $X=\bigcup_{i\in J} O_i$.
\end{defi}

\begin{prop}\label{KCOMP}
If $\kappa=|L|$, for every finite tuple $x$ of variables, the spectral topology on $S_x(T)$ is $\kappa^+$-compact.
\end{prop}
\begin{proof}
Suppose that $S_x(T)=\bigcup_{i\in I} O_i$, with $O_i$ a spectral open for each $i$. By definition of the spectral topology, for every $i\in I$ there exists a family $(\phi^i_j(x))_{j\in J_i}$ of positive formulas such that $O_i=\bigcup_{j\in J_i} [\phi^i_j(x)]$. This means we have $S_x(T)=\bigcup \{[\phi^i_j(x)] : (i,j)\in \bigcup_{i\in I} \{i\}\times J_i\}$. Now the set of positive $L$-formulas with free variables among $x$ has cardinality $\kappa^{<\omega}=\kappa$, hence we may choose a subset $K\subset\bigcup_{i\in I} (\{i\}\times J_i)$ such that $|K|\leq \kappa$ and $S_x(T)=\bigcup_{(i,j)\in K} [\phi^i_j]$. Let $I'$ be the set of all $i\in I$ such that $(i,j)\in K$ : we have $|I'|\leq |K|$. For every $(i,j)\in K$, we have $[\phi_j^i]\subset O_i$, whence $S_x(T)=\bigcup_{(i,j)\in K} [\phi^i_j]=\bigcup_{i\in I'} O_i$, and the proof is complete. 
\end{proof}

Remember that a topological space $E$ has the \emph{Baire property}, if for every countable family $(O_i)_{i<\omega}$ of dense subsets of $E$, the intersection $\bigcap_{i<\omega} O_i$ is still a dense subset of $E$.
\begin{prop}
The positive type spaces $S_x(T)$ have the Baire property for the spectral topology.
\end{prop}
\begin{proof}
Let $(U_i:i<\omega)$ be a countable family of dense spectrally open subsets of $S_x(T)$ and $O$ a non empty spectrally open subset of $S_x(T)$. As $O\neq\emptyset$, by definition of the spectral topology we may find a formula $\phi_0(x)$ such that $\emptyset\neq[\phi_0]\subset O$. Suppose by induction hypothesis that $n<\omega$ and we have found $(\phi_i(x):i\leq n)$ such that $\emptyset\neq[\phi_j]\subset [\phi_i]$ for $j\geq i$ and $[\phi_n]\subset O\cap(\bigcap_{i<n} U_i)$. As $U_n$ is dense and $[\phi_n]$ is open for the spectral topology, there exists a formula $\phi_{n+1}(x)$ such that $\emptyset\neq[\phi_{n+1}]\subset [\phi_n]\cap U_n\subset O\cap(\bigcap_{i<n+1} U_i)$. By induction, we find a decreasing family $([\phi_i])_{i<\omega}$ of basic definably closed sets : as the definable topology is compact, the intersection of the family is non empty and included in $O\cap(\bigcap_{i<\omega} U_i)$, which is then non empty, and $\bigcap_{i<\omega} U_i$ is dense : the spectral topology has the Baire property.
\end{proof}

Boolean combinations of positive formulas define subsets of positive type spaces, in the following way. For $\phi(x),\psi(x)$ positive, we know that $[\psi(x)]^c$ is a spectral open of $S_x(T)$, hence we may define $[\neg\psi]$ as $[\psi]^c$, and $[\phi\wedge\neg\psi]=[\phi]\cap [\neg\psi]$, a spectrally open set. As every Boolean combination $\chi$ of positive formulas is formally equivalent to a finite disjunction of formulas of the form $\phi\wedge\neg\psi$ with $\phi$ and $\psi$ positive, every such Boolean combination defines a subset of $S_x(T)$, which is spectrally open (and one checks that this does not depend on the representation of $\chi$).

\begin{defi}
Say that a formula $\phi$ of $L$ is \emph{constructible}, if it is a Boolean combination of positive formulas. Likewise, say that a subset $S$ of $S_x(T)$ is \emph{constructible}, if there exists a constructible formula $\phi(x)$ of $L$ such that $S=[\phi(x)]$.
\end{defi}

Now by definition every spectrally open subset of $S_x(T)$ is a union of constructible sets, and every constructible set is a spectral open, so the collection of constructible subsets is a basis for the spectral topology on $S_x(T)$. The notion of resultant may be extended to all constructible formulas in the context of positive logic, using the Kaiser hull of $T$.

\begin{defi}
If $\phi(x)$ is a constructible formula, define the \emph{constructible resultant of $\phi$ modulo $T$} as the set $Res_T^c(\phi)$ of all constructible formulas $\psi(x)$ such that $T_k\mo \forall x\ \phi\wedge\psi\imp \bot$.
\end{defi}

\begin{rem}
The constructible resultant is defined relatively to $T_k$, as compared to the (positive) resultant, defined relatively to $T_u$, because the sentences expressing the incompatibility of two constructible formulas are not h-universal in general, while they may be encoded in $T_k$.
\end{rem}

\begin{prop}
For every constructible formula $\phi(x)$, we have $[\phi(x)]^c=\bigcup\{[\psi(x)]:\psi\in Res_T^c(\phi)\}$.
\end{prop}
\begin{proof}
We proceed by induction on the complexity of the formula $\phi$.\\
- If $\phi$ is atomic, then we have $[\phi]^c=\bigcup \{[\psi] : \psi\in Res_T(\phi)\}\subset \bigcup\{[\psi] : \psi\in Res_T^c(\phi)\}\subset [\phi]^c$, by definition of the resultant, whence the equality.\\
- If $\phi=\neg\psi$, we have $[\phi]^c=[\psi]\subset \bigcup\{[\theta] : \theta\in Res_T^c(\phi)\}$ (because $\psi\in Res_T^c (\neg\psi)$ !) $\subset [\phi]^c$, whence the equality.\\
- If $\phi=\psi\vee\chi$ and $(\theta,\lambda)\in Res_T^c(\psi)\xx Res_T^c(\chi)$, it is easy to check that $\theta\wedge\lambda\in Res_T^c(\psi\vee\chi)$, hence we get $[\phi]^c=[\psi]^c\cap [\chi]^c=\bigcup\{[\theta] : \theta\in Res_T^c(\psi)\}\cap \bigcup\{[\lambda] : \lambda\in Res_T^c(\chi)\}$ (by induction hypothesis) $\subset \bigcup\{[\zeta] : \zeta\in Res_T^c(\psi\vee\chi)\}\subset [\phi]^c$, whence the equality.\\
- If $\phi=\psi\wedge\chi$, we have $[\psi\wedge\chi]^c=[\bigvee Res_T^c(\psi)]\cup [\bigvee Res_T^c(\chi)]$ (by induction hypothesis) $\subset [\bigvee Res_T^c(\psi\wedge\chi)]\subset [\psi\wedge\chi]^c$, whence the equality.
\end{proof}

The end of this section illustrates how the spectral topology embraces a mild treatment of the negation in positive logic, in the form of constructible formulas and sets. Now every spectral open is by definition a union of "positively definable" subsets, and the constructible subsets form a finer basis of topology than the positively definable ones. The question arises of comparising constructible subsets to spectral opens. In general, there are more spectrally open than constructible subsets, otherwise every spectral open would have a spectrally open complement, which fails for instance in every positively model complete theory in which not every type definable set is definable. We discuss an example where we invoke classical model theory.

\begin{ex}
Consider the theory $T$ of real fields in the language $L=(+,\xx,-,0,1)$ of rings, the positive models of which are the real closed fields, which are axiomatisable by a positive model complete theory $T^*$; notice that here the spectral and the definable topologies coincide by Proposition \ref{SPEC}. It is easy to see that in $\RR$, Cantor's triadic set $\mf K$ is (positively) type definable without parameters, by a partial positive type $\pi(x)$ in one variable. This type defines in $S_1(T)$ a definably closed set $[\pi(x)]$, which is then spectrally closed. Now suppose that $[\pi(x)]$ is spectrally open in $S_1(T)$ : there exists a family $(\phi_i(x))_I$ of formulas in one variable such that $[\pi(x)]=\bigcup_{i\in I} [\phi_i]$, and as there are only $|L|=\omega$ such formulas, we may suppose that $I$ is at most countable. As $\RR$ is a real closed field, we have $\mf K=\bigcup_{i<\omega} \phi_i(x)^\RR$ and as $\RR$ is o-minimal as considered as a classical $L$-structure, every $\phi_i(x)^\RR$ is a finite union of intervals, hence $\mf K$ is a countable union of intervals $\mf K=\bigcup_{i<\omega} I_i$. As $|\mf K|>\omega$, one of those contains at least two points, contradicting the fact that $\mf K$ has empty interior. This means that $S_1(T)-[\pi(x)]$ is open but not constructible.
\end{ex}

\section{Geometric formulas and types}\label{GEOM}
In this section we study the relationships between the spectral topology and geometric logic, and show that in this context the analogues of partial positive types define all the subsets of $S_x(T)$. From now on, we set $\kappa=|L_{\omega\omega}|$ and refer to the introduction for the definition of the class $L_{\oo\omega}^g$ of geometric formulas.

\begin{lem}[Disjunctive normal form]
Every geometric formula is logically equivalent to a disjunction of positive primitive formulas. In particular, every geometric formula is equivalent to one in $L_{2^\kappa\omega}$.
\end{lem}
\begin{proof}
For atomic formulas, there is nothing to prove. Suppose that $(\phi_i:i<m)$ is a finite set of geometric formulas, each equivalent to a disjunct $\bigvee \Phi_i$ of a set of p.p. formulas. The formula $\bigwedge_{i<m} \phi_i$ is equivalent to $\bigvee\{ \bigwedge_{i<m} \psi_i : \psi_i\in \Phi_i, i<m\}$. Suppose that $\phi$ is logically equivalent to a disjunct $\bigvee \Phi$ of a set of p.p. formulas : if $x$ is any variable, the formula $\exists x\phi$ is logically equivalent to the disjunct $\bigvee\{\exists x\psi : \psi\in \Phi\}$ of p.p. formulas. Finally, if $\Phi$ is a set of geometric formulas, each $\phi\in \Phi$ being equivalent to a disjunct $\bigvee \Psi_\phi $ of a set of p.p. formulas, the disjunct $\bigvee \Phi$ is logically equivalent to the disjunct $\bigvee(\bigcup\{ \Psi_\phi :\phi\in \Phi\})$. This proves the existence of a disjunctive form.\\
Now as the set of positive primitive formulas of $L$ has cardinality at most $\kappa=|L|$, any set $\Phi$ of p.p. formulas has cardinality at most $\kappa$, hence $\bigvee \Phi\in L_{2^\kappa\omega}$, which establishes the second part.
\end{proof}

\begin{defi}
(i) A \emph{disjunctive normal form} of a geometric formula is an equivalent disjunction of p.p. formulas, by the lemma.\\
(ii) We will say that a geometric formula is \emph{normal}, if it is a disjunctive normal form.
\end{defi}

By the lemma, the class of all geometric formulas $\phi(x)$ satisfied in an $L$-structure $M$ by a point $a\in M_x$ is determined by the set of normal geometric formulas $\phi(x)$ such that $M\mo \phi(a)$.

\begin{defi}
(i) If $M$ is an $L$-structure, $x$ is a finite tuple of variables and $a\in M_x$, the \emph{geometric type of $a$ in $M$}, noted $tp^g_M(a)$, will denote the set of normal geometric formulas $\phi(x)$ such that $M\mo \phi(a)$.\\
(ii) More generally, by a \emph{geometric type in variables $x$}, we will mean a set $\pi(x)$ of geometric formulas; we will call it \emph{normal} if it contains normal formulas only.
\end{defi}

Now let $p(x)$ be a positive type over $T$, and $a,b$ two realisations of $p$ in positive models $M,N$ of $T$. If $\phi(x)=\bigvee \Phi(x)$ is a normal geometric formula such that $M\mo \phi(a)$, there exists $\psi\in \Phi$ such that $M\mo \psi(a)$, and hence $\psi\in p$, which means that $N\mo \psi(b)$, and $N\mo \phi(b)$. In other words, every positive type $p(x)$ determines a unique set $p^*(x)$ of normal geometric formulas. We complete this remark as the following

\begin{prop}\label{TYPGEO}
The map $p\mapsto p^*$, which maps a positive type $p(x)$ to the geometric type of any of its realisations, is a bijection from $S_x(T)$ to the set of normal geometric types which are consistent with $T$ and maximal with this property.
\end{prop}
\begin{proof}
First, if $p^*(x)\subset \pi(x)$, a geometric type consistent with $T$, let $a\in M_x$ be a realisation of $\pi$ in a model $M$ of $T$. As the validity of geometric sentences with parameters is preserved under homomorphisms, we may suppose that $M$ is a positive model of $T$. Let $\phi\in \pi$ : there exists a set $\Phi$ of p.p. formulas such that $\phi=\bigvee \Phi$, so there is $\psi\in \Phi$ such that $M\mo \psi(a)$, whence $\psi\in p$ by maximality of $p$, so $\phi\in p^*$, and $p^*=\pi$ is maximal as a geometric type consistent with $T$.\\
Secondly, if $q$ is a maximal geometric type consistent with $T$, realise $q$ by a point $a\in M_x$ of a positive model $M$ of $T$ : the positive type $p=tp^+_M(a)$ is an element of $S_x(T)$ such that $p^*=q$ by definition of $p^*$, so the map is surjective. Now if $p,q\in S_x(T)$ and $p^*=q^*$, any realisation of $p^*$ is a realisation of both $p$ and $q$, hence $p=q$ by their maximality, and the map is injective.
\end{proof}

All this means that any geometric formula $\phi(x)$ defines a subset of $S_x(T)$, by $[\phi(x)]=\{p\in S_x(T) : p^*\mo \phi\}$, in which by $p^*\mo \phi$ we mean that for some (any) normal form $\psi$ of $\phi$, we have $\psi\in p^*$. Now the following (the second part of which could alternatively be proved using Lemma 17 of \cite{BYP}) should be obvious.

\begin{cor}
The subsets of the form $[\phi(x)]$, for a geometric formula $\phi(x)$, are exactly the spectrally open subsets of $S_x(T)$, and for every finite tuple $y$ of variables, the canonical projection map $\pi:S_{xy}(T)\to S_x(T)$ is open for the spectral topology.
\end{cor}
\begin{proof}
Using a disjunctive normal form $\bigvee \Phi(x)$ for a geometric formula $\phi$, we see that $[\phi(x)]=\bigcup \{[\psi] : \psi\in \Phi\}$, hence is open for the spectral topology. Conversely, if $O$ is open for the spectral topology, find a family $(\phi_i(x))_I$ of positive formulas such that $O=\bigcup_I [\phi_i(x)]$ : we have $O=[\bigvee\{\phi_i:i\in I\}]$, and as $\bigvee\{\phi_i:i\in I\}$ is geometric, this is enough for the first part.
As for the second, if $[\phi(x,y)]$ is spectrally open in $S_{xy}(T)$ with $\phi$ geometric, it is easy to check that $\pi([\phi(x,y)])$ is the open subset defined by the geometric formula $\exists y\phi(x,y)$.
\end{proof}

For every positive type $p\in S_x(T)$, the singleton $\{p\}=\bigcap\{[\phi(x)]\in L^+: \phi\in p\}$ is closed for the definable topology, hence for the spectral topology. Explicitly, we may describe its spectrally open complement $S_x(T)-\{p(x)\}=\bigcup_{\phi\in p} \bigcup_{\psi\in Res_T(\phi)} [\psi(x)]$. Now the spectral topology of $S_x(T)$ may itself be construed as a collection of basic closed sets for a certain topology, which is analogous to the definable topology, in the sense that the basic closed sets have the form $[\phi(x)]$ for any geometric formula $\phi(x)$. What precedes shows that every singleton is open for this topology, which is then trivial. In other words, we have the

\begin{prop}\label{GEODEF}
For every subset $X$ of $S_x(T)$, there exists a geometric type $\pi(x)$ such that $X=[\pi(x)]$, where $[\pi(x)]=\{p\in S_x(T) : \forall \phi\in \pi,\ p^*\mo\phi\}$.
\end{prop} 

In particular, the geometric types have a uniform complement in positive models of $T$. For a geometric type $\pi(x)$, there exists an alternative description of this complement, which goes as follows. Let $\{\pi_i(x):i\in I\}$ be a family of geometric types. We define the \emph{disjunction} $\bigvee_I \pi_i(x)$ as the geometric type $\{\bigvee_{i\in I} \phi_i(x):\forall i\in I,\phi_i\in \pi_i\}$.

\begin{lem}\label{DIJTYP}
The type $\bigvee_I \pi_i$ is logically equivalent to the disjunction of the $\pi_i$'s.
\end{lem}
\begin{proof}
Write $\pi(x)=\bigvee_I \pi_i(x)$. Let $M$ be an $L$-structure and $a\in \bigcup_I \pi_i(x)^M$ : there is $i_0\in I$ such that $a\in \pi_{i_0}(x)^M$, so let $\phi(x)=\bigvee_I \phi_i(x)\in \pi(x)$ : we have $M\mo \phi_{i_0}(a)$, so $M\mo\phi(a)$, hence $a\in\pi(x)^M$. Reciprocally, suppose that $a\notin \bigcup_I\pi_i(x)^M$ and let $i\in I$ : we have $a\notin \pi_i(x)^M$, so there is $\phi_i(x)\in \pi_i(x)$ such that $M\not\mo \phi_i(a)$. This means that $M\not\mo \bigvee_I \phi_i(a)$, hence $a\notin \pi(x)^M$, because this last formula $\bigvee_I \phi_i(x)$ is in $\pi(x)$.
\end{proof}

\begin{prop}\label{GEOCOMP}
Every geometric type $\pi(x)$ has a uniform geometric complement in positive models of $T$ (i.e., there exists a geometric type $\neg_T\pi(x)$ such that $\neg_T\pi(x)^M$ defines the complement of $\pi(x)^M$ in every positive model $M$ of $T$).
\end{prop}
\begin{proof}
We start with a geometric formula $\phi(x)$. Up to logical equivalence, we may suppose that $\phi(x)=\bigvee_I \phi_i(x)$ is normal, i.e. with $\phi_i$ positive primitive for each $i\in I$. If $i\in I$, let $\Phi_i(x)=Res_T(\phi_i(x))$, a set of positive formulas : the formula $\bigvee\Phi_i(x)$, which we note $\neg_T\phi_i(x)$, is geometric. Let $\neg_T\phi(x)=\{\neg_T\phi_i(x):i\in I\}$ ; this is a geometric type, and if $M\mop T$, we have $M_x-\phi(x)^M=M_x-\bigcup_{i\in I} \phi_i(x)^M=\bigcap_{i\in I} (M_x-\phi_i(x)^M)=\bigcap_{i\in I} \neg_T\phi_i(x)^M$ (by the properties of the resultant) $=(\neg_T\phi(x))^M$, so the geometric type $\neg_T\phi(x)$ defines the complement of the geometric formula $\phi(x)$ in every positive model of $T$.\\
If now $\pi(x)$ is a geometric type, we define $\neg_T\pi(x)$ as the geometric type $\bigvee \{\neg_T\phi(x) : \phi(x)\in \pi(x)\}$, and we contend that $\neg_T\pi$ defines the complement of $\pi$ in every positive model of $T$. Indeed, if $M\mop T$, we have $M_x-\pi(x)^M=M_x-\bigcap_{\phi\in \pi}\phi(x)^M= \bigcup_{\phi\in \pi} (M_x-\phi(x)^M)=\bigcup_{\phi\in\pi} \neg_T\phi(x)^M$ (by what precedes on the complements of geometric formulas) $=\neg_T\pi(x)^M$. 
\end{proof}

In conclusion, the subsets of positive type spaces are essentially "abstract definitions" for subsets of positive models of $T$ defined by geometric types. In particular, geometric types are closed under arbitrary disjunctions and conjunctions, as well as negations and existential quantifications. With these properties they must interpret all the formulas in $L_{\oo\omega}$, which we make precise in the following section.

\section{Existential models and infinitary logic}\label{EXINF}

We adapt the notion of an existentially universal structure, defined originally in terms of existential types in the context of inductive classes of extensions, to the present context of inductive classes of homomorphisms and partial positive types (see \cite{FAD}, section 1.2).

\begin{defi}
Say that a structure $M$ in a class $\C$ is \emph{(positively) existentially universal (in $\C$)}, if for every partial positive type $\pi(x,y)$ in finitely many variables and every point $b\in M_y$ such that $\pi(x,b)$ is realised in a continuation of $M$ in $\C$, there is $a\in M_x$ such that $M\mo \pi(a,b)$. 
\end{defi}

\begin{rem}
Any existentially universal structure in $\C$ is positively existentially closed; this is a strengthening of the notion.
\end{rem}

\begin{prop}
If $\C$ is an inductive class, then every structure in $\C$ has a continuation into an existentially universal structure of $\C$.
\end{prop}
\begin{proof}
It suffices to reproduce the construction of a positively existentially closed continuation, replacing positive formulas by partial positive types in the proof of Fact \ref{INDPEC} found in \cite{BYP}.
\end{proof}

\begin{defi}
(i) If $T$ is an h-inductive theory, say that a model $M$ of $T$ is \emph{existential}, noted $M\moe T$, if it is a positively existentially universal structure of $\M(T)$. Any existential model is positive.\\
(ii) If $M$ is an $L$-structure and $f:M\to N$, we say that $f$ is an \emph{existential extension of $M$}, if $(N,f)$ is an existential model of $T(M|M)$ (or equivalently of $T_u(M|M)$).  
\end{defi}

\begin{rem}
(i) An $L$-structure $M$ is an existential model of $T$ if and only if it is an existential model of $T_u$, so $T$, $T_u$ and $T_k$ have the same existential models.\\
(ii) If $T$ has the "joint continuation property" (i.e. any two models have a common continuation into a third), the existential models of $T$ are the positively $\omega$-saturated positive models of $T$ introduced in \cite{BYP}.\\
(iii) For any $L$-structure $M$, an existential model $N$ of $T(M|M)$ is a positive extension which realises all finitary positive types over $Ma$, for every point $a$ from $N$.\\
(iv) By compactness, a universal domain $U$ of cardinal at least $|L|^+$ for a $\Pi$-theory $T$ in the positive fragment of positive formulas, as introduced in Definition 2.11 of \cite{PMTCAT}, is an existential model of $T$.
\end{rem}

We recall that by Lemma 18 of \cite{BYP}, any two $L$-structures $M,N$ with the same h-universal theory $U$ (without parameters), have the "joint continuation property", i.e. there exist two homomorphisms $f,g:M,N\to P$ into a common codomain which is a model of $U$. Frow this we get the following

\begin{lem}\label{CONTPOS}
Let $M,N\mo T$ and $a\in M_x$, $b\in N_x$ such that $tp^+_M(a)=tp^+_N(b)$. There exist $P\mo T$ and $f:M\to P$, $g:N\to P$, such that $f(a)=g(b)$.
\end{lem}
\begin{proof}
It is possible to restate the hypothesis as $T_u(M|a)=T_u(M|b)$, because the h-universal formulas satisfied by $a$ in $M$ and $b$ in $N$ are exactly the negations of the formulas in their positive type. If $c$ is an appropriate new tuple of constant symbols, the theory $T_u(M|a)(c/a)$ is complete in the language $L(M\cup c)$ so it has the joint continuation property : there exists a model $(P,d)\mo T_u(M|a)(c/a)$ and two $L(M\cup c)$-homomorphisms $f:(M,a)\to (P,d)$ and $g:(N,b)\to (P,d)$. As $T_u\subset T_u(M|a)$, we have $P\mo T_u$ and $f(a)=g(b)$ : continuing if necessary, we may suppose that $P$ is a model of $T$.
\end{proof}

\begin{thm}
If $M$ and $N$ are two existential models of an h-inductive theory $T$ and $T_u(M)=T_u(N)$, then $M$ and $N$ are infinitely equivalent. 
\end{thm}
\begin{proof}
We build a back-and-forth system of partial isomorphisms between $M$ and $N$ and use Karp's theorem (Corollary 3.5.3 of \cite{WH}). Let $S$ be the collection of all pairs $(a,b)$ of finite tuples with $a$ from $M$ and $b$ from $N$, with the same sorting, and such that $tp^+_M(a)=tp^+_N(b)$. By hypothesis, we have $(\emptyset,\emptyset)\in S$, because $M$ and $N$ satisfy the same h-universal, and hence the same positive sentences; thus, $S$ is not empty. Let $(a,b)\in S$ and $\alpha\in M_y$, for a variable $y$, and let $p(x,y)=tp^+_M(a,\alpha)$ : by definition of $S$ we have $tp^+_M(a)=tp^+_N(b)$, which means that $T_u(M|a)=T_u(N|b)$, and by the preceding lemma there exists a common continuation of $M$ and $N$ by $f$ and $g$ into a model $P$ of $T$, with $f(a)=g(b):=c$. The (complete) positive type $p(b,y)$ is realised by $f\alpha$ in $(P,f)$, so as $N$ is existential there exists a realisation $\beta\in N_y$ of $p(b,y)$ in $N$; as $p(b,y)$ is complete, we have $tp^+_M(a\alpha)=p(xy)=tp^+_N(b\beta)$, so $(a\alpha,b\beta)\in S$. By symmetry, if we choose $\beta\in N_y$ we may find $\alpha\in N_y$ such that $(a\alpha,b\beta)\in S$, which is thus a back-and-forth system, and $M\equiv_\oo N$ by Karp's theorem.
\end{proof}

\begin{rem}
(i) If $T$ is a complete h-universal theory and $M$ is an existential model of an h-inductive theory $T'\subset T$, then $T'_u\subset T_u$ and $M$ is an existential model of $T$.\\
(ii) By the theorem, if $M$ and $N$ have the same h-universal theory $T$ of which they are existential models, they are infinitely equivalent.
\end{rem}

\begin{cor}
Let $T$ be an h-inductive theory. If $M,N\moe T$, $a,b\in M_x,N_x$ and $tp^+_M(a)=tp^+_N(b)$, then 
for every formula $\phi(x)\in L_{\oo\omega}$, we have $M\mo \phi(a)$ if and only if $N\mo \phi(b)$.
\end{cor}
\begin{proof}
Let $c$ be a new tuple of constants with the same sorting as $a$ and $b$ : the $L(c)$-structures $(M,a)$ and $(N,b)$ are existential models of $T$ in this extended language, and the hypothesis says that $T_u(M|a)=T_u(N|b)$ (as before, because $a$ and $b$ satisfy the same positive sentences); by the preceding theorem we have $(M,a)\equiv_\oo (N,b)$. Now if $\phi(x)\in L_{\oo\omega}$, the formula $\phi(c/x)$ is in $L(c)_{\oo\omega}$, so $M\mo \phi(a)$ if and only if $M\mo \phi(b)$. 
\end{proof}

Let now $p(x)\in S_x(T)$ be a finitary positive type over $T$ : if $a$ and $b$ are two realisations of $p$ in existential models $M$ and $N$ of $T$, we have $tp^+_M(a)=tp^+_N(b)$ and by the corollary, for every formula $\phi(x)\in L_{\oo\omega}$ we have $M\mo \phi(a)$ if and only if $N\mo \phi(b)$, so the formulas of $L_{\oo\omega}$ satisfied by a realisation of $p$ in an existentiel model of $T$ do not vary with the realisation. We may thus define $[\phi(x)]=\{p\in S_x(T) : \forall M\moe T, \forall a\in M_x,a\mo_M p\imp M\mo \phi(a)\}$ for any such formula. Now if $\pi(x)$ is a geometric type such that $[\phi(x)]=[\pi(x)]$ by Proposition \ref{GEODEF}, for every existential model $M$ of $T$ and $a\in M_x$, we have $M\mo \phi(a)\iff tp^+_M(a)\in [\phi(x)] \iff M\mo \pi(a)$, which means we have the

\begin{cor}\label{INFGEO}
For every formula $\phi(x)\in L_{\oo\omega}$, there exists a geometric type $\pi_\phi(x)$ such that for every existential model $M$ of $T$ and every $a\in M_x$, one has $M\mo \phi(a)$ if and only if $M\mo \pi_\phi(a)$.
\end{cor}

\begin{cor}\label{EXTINF}
If $M,N\moe T$ and $f:M\to N$, then $f$ is $\oo$-elementary.
\end{cor}
\begin{proof}
Let $\phi(x)\in L_{\oo\omega}$ and $m\in M_x$. By the preceding corollary we find a geometric type $\pi_\phi(x)$ equivalent to $\phi(x)$ in existential models of $T$. If $M\mo \phi(m)$, we thus have $M\mo \pi(m)$, whence $N\mo \pi(fm)$ (because the validity of geometric types is preserved under homomorphisms) and $N\mo \phi(fm)$ : $f$ is $\oo$-elementary.
\end{proof}

We step back to the classical setting. If $U$ is any finitary first order theory, let $U^G$ be its positive Morleyisation in the language $L^G$ (see \cite{BYP}). We notice that positive Morleyisation does not change the size of the language, i.e. $|L^G|=|L|=\kappa$.

\begin{cor}\label{INFEX}
For every formula $\phi(x)\in L_{\oo\omega}$, there exists a formula $\phi'(x)$ of $L_{2^{2^\kappa}\omega}$ such that $\phi'$ is equivalent to $\phi$ in every $\omega$-saturated $L$-structure.
\end{cor}
\begin{proof}
We note $\lambda=2^{2^\kappa}$. Let $U=\emptyset$. Every formula $\phi(x)$ of $L^G_{\oo\omega}$ is logically equivalent in existential models of $U^G$ to a normal geometric type $\pi(x)$. There are at most $\lambda$ non equivalent such types, each of one having stricltly less than $\lambda$ elements, so $\phi(x)$ is logically equivalent, in existential models of $U^G$, to a formula of $L^G_{\lambda\omega}$, the conjunction of the type $\pi(x)$. In particular, this is true for every formula $\phi(x)$ of $L_{\oo\omega}$, and every geometric formula $\psi=\bigvee \Phi_\psi$ in $\pi_\phi$ is logically equivalent modulo $U^G$ to a formula $\psi'=\bigvee \{\theta':\theta\in \Phi_\psi\}$, where $\theta'\in L$ for each $\theta\in \Phi_\psi$. % ("de-Morleyisation").
Now if $M$ is any $\omega$-saturated $L$-structure, $M^G$ is an existential model of $U^G$, hence $\phi$ is equivalent in $M$ to $\bigwedge_{\psi\in \pi}\bigvee_{\theta\in \Phi_\psi}  \theta'$, a formula of $L_{\lambda\omega}$.
\end{proof}

This corollay means that we may define the "infinitary type" of a finite tuple as a \emph{set} of formulas, if we allow ourselves to work in $\omega$-saturated structures or more generally in existential models in posiive model theory, because we may assign a bound to the complexity of the formulas we need.\\
These considerations about other kinds of logical operations in the context of positive model theory were originally motivated by an attempt to introduce some elements of the classical setting, in order to take care of other formulas than positive ones.
The "constructible" formulas introduced in section \ref{SPECTOP} are interpreted as spectral open subsets of positive type spaces. As the Stone space of ultrafilters of this Boolean algebra of subsets of $S_x(T)$ is compact for the usual topology, a most natural question is the "semantic" meaning of this compactness. Now the natural order on this Boolean algebra is given by $[\phi(x)]\leq [\psi(x)]\iff$ for every positive model $M$ of $T$, $M\mo \forall x\ \phi(x)\imp\psi(x)$. This last sentence being h-inductive, this is equivalent to saying that $"\forall x\ \phi(x)\imp\psi(x)"\in T_k$, the Kaiser hull of $T$. In other words, the Boolean algebra of constructible subsets of $S_x(T)$ reflects the logic of models of $T_k$ and not of positive models of $T$, and the ultrafilters are the "constructible types" of tuples in models of $T_k$. As such they generalise the positive types, but as the satisfaction of constructible sentences is not stable under homomorphisms in general, there is little hope for this to be useful in positive model theory, at least in general.
\\Another Boolean algebra of possible interest would be the algebra of regular spectrally (or alternatively definably) open subsets, with its compact space of ultrafilters, the semantic meaning of which is much less clear, in particular because we do not know at present if positive types induce ultrafilters of regular opens (in both cases). We leave it as a distinct problem to be adressed in further research.

\section{Geometric Morleyisation}\label{GEOMOR}
The existence of positively existentially closed and existentially universal structures in a full subcategory $\C$ of $\LS$ is secured whenever $\C$ is closed under directed colimits. This would allow one to expand the basic considerations of positive model theory to the more general "geometric theories", which we will call here "g-inductive", in order to keep with the habit of not distinguishing between formulas and axioms in set-theoretic model theory. We will say that a formula $\phi(x)$ in $L_{\oo\omega}$ is \emph{g-inductive}, if it has the form $\forall y\ \psi(x,y)\imp \chi(x,y)$, where $\psi$ and $\chi$ are geometric; we may always suppose that $\psi$ and $\chi$ are normal, hence any class of g-inductive sentences is equivalent to a set of g-inductive sentences from $L_{2^\kappa\omega}$.
Without loss of generality, we then define a \emph{g-inductive theory} as a set of g-inductive sentences and one easily checks the following

\begin{lem}\label{INDGEO}
The full subcategory $\M(T)$ of models of a g-inductive theory $T$ is closed under directed colimits.
\end{lem}

One could try and reproduce the concepts arising in positive model theory, as positive models, positive types... for such g-inductive theories. However, the lack of compactness for infinitary logic would throw out every result which essentially needs it in this context, so this would seem to be a different kind of study. Nevertheless, it is possible to axiomatise the positive models of an h-inductive theory by a g-inductive one, hence positive model theory is in some sense a particular case of geometric logic, in which some of the first order compactness is "retained".\\
In this short and last section we will rather study the interplay between infinitary logic and g-inductive theories, as sketched in the precedent section with Corollary \ref{INFEX} which simplifies the formulas of $L_{\oo\omega}$ in $\omega$-saturated $L$-structures; this might have been proved directly, but we have used positive Morleyisation, showing that this is essentially a positive model theoretic property.
It is known that infinitary theories may be translated in a very general way into "basic ones" (see \cite{MP}, Proposition 3.2.8); however, we have not read about the reduction of $L_{\oo\omega}$ to geometric logic by this kind of Morleyisation, so we adress this question here, restricting ourselves to a theory $T$ in $L_{\infty\omega}$, for a finitary first order language $L$. As $T$ is a \emph{set} there exists a %regular 
cardinal $\lambda$ such that $T\subset L_{\lambda\omega}$ and we may adapt positive Morleyisation to this fragment in the following simple way.

\begin{defi}
For every formula $\phi(x)$ of $L_{\lambda\omega}$, we introduce a new relational symbol of the same arity noted $R_\phi(x)$, and we get an extended first order language $L^G$ with the same sorts. We then define the theory $T^G$ in $L^G$, which contains the following sentences of $L^G$ :\\
(i) If $\phi(x)$ is atomic in $L$, two g-inductive sentences expressing $\forall x\ \phi(x)\iff R_\phi(x)$ \\
(ii) If $\Phi(x)$ is a $\lambda$-small set of formulas of $L_{\lambda\omega}$, two g-inductive sentences expressing $\forall x\  R_{\bigvee \Phi}(x) \iff \bigvee \{ R_\phi:\phi\in \Phi\}(x) $\\
(iii) If $\phi(x)$ is in $L_{\lambda\omega}$, the sentences $\forall x\ (R_\phi\wedge R_{\neg\phi})(x)\imp \bot$ and $\forall x\ \top \imp (R_{\phi}\vee R_{\neg\phi})(x))$, expressing $\forall x\ \neg R_\phi(x)\iff R_{\neg\phi}(x)$\\
(iv) If $\Phi(x)$ is a $\lambda$-small set of formulas of $L_{\lambda\omega}$, the sentences $\forall x\ (R_{\bigwedge\Phi}\wedge \bigvee\{R_{\neg\psi}:\psi\in \Phi\})(x)\imp \bot$ and $\forall x\ \top\imp (R_{\bigwedge\Phi}\vee \bigvee\{R_{\neg\psi} : \psi\in \Phi\})(x)$, expressing $\forall x\ R_{\bigwedge\Phi}(x)\iff \neg\bigvee \{R_{\neg\psi} : \psi\in \Phi\}(x)$\\
(v) If $\phi(x,y)$ is in $L_{\lambda\omega}$, two sentences expressing $\forall x\ R_{\exists y\phi}(x)\iff \exists y R_\phi(x,y)$\\
(vi) If $\phi$ is a sentence in $T$, the propositional constant $R_\phi$.\\
The theory $T^G$ is g-inductive, and we call it the \emph{geometric Morleyisation of $T$}.
\end{defi}

\begin{prop}
(i) Every model $M$ of $T$ has a canonical expansion $M^G$ which is a model of $T^G$, and the $L$-reduct of every model of $T^G$ is a model of $T$\\
(ii) The categories $\M_{\lambda\omega}(T)$ (of $\lambda$-elementary extensions of models of $T$) and $\M(T^G)$ (of homomorphisms of models of $T^G$) are isomorphic.
\end{prop}
\begin{proof}
(i) Let $M$ be a model of $T$. For every formula $\phi(x)\in L_{\lambda\omega}$, we interpret $R_\phi(x)$ in $M$ as $\phi(x)^M$ : call this expansion $M^G$. By definition of $T^G$ one checks that $M^G$ is a model of the axioms introduced in clauses (i)-(v) of the definition of $T^G$. As for clause (vi), if $\phi\in T$ is a sentence, the propositional constant $R_\phi$ is in $L^G$, and is interpreted in $M^G$ as $\phi^M$, which is $M^\emptyset=\{\emptyset\}$, i.e. $M^G\mo R_\phi$, so $M^G\mo T^G$.\\
Conversely, if $M\mo T^G$, we must prove that $M|L\mo T$, and for this we show that $M\mo R_\phi(a)$ if and only if $M|L\mo \phi(a)$, for every sentence with parameters $\phi(a)$ with $\phi(x)\in L_{\lambda\omega}$ and $a\in M_x$, by induction on the complexity of $\phi$ :\\
- If $\phi$ is atomic, by definition of $T^G$ we have $M\mo \forall x\ R_\phi(x) \iff \phi(x)$ (clause (i) of the definition of $T^G$); we then have $M|L\mo \phi(a)\iff M\mo \phi(a) \iff M\mo R_\phi(a)$\\
- If $\phi=\bigvee \Phi$, where $\Phi$ is a set of cardinality $|\Phi|<\lambda$ of formulas of $L_{\lambda\omega}$, we have $M|L\mo \phi(a)\iff \exists \psi\in \Phi,M|L\mo \psi(a) \iff$ (by induction hypothesis) $\exists \psi\in \Phi,M\mo R_\psi(a)\iff$ (by clause (ii) of the definition of $T^G$) $M\mo R_\phi(a)$\\
- If $\phi=\neg\psi$, by induction hypothesis and clause (iii) of the definition of $T^G$, we have $M|L\mo \phi(a) \iff M|L\not\mo \psi(a) \iff M\not\mo R_\psi(a)\iff M\mo R_\phi(a)$\\
- If $\phi=\bigwedge\Phi$, we have $M|L\mo \phi(a)\iff \forall \psi\in \Phi, M|L\mo \psi(a)\iff$ (by induction hypothesis) $\forall \psi\in \Phi, M\mo R_\psi(a)\iff $ (by the $\neg$ clause) $\not\exists \psi\in \Phi, M\mo R_{\neg\psi}(a)\iff $ (by the $\bigwedge$ clause (iv)) $M\mo R_{\phi}(a)$.\\
- If $\phi=\exists y\psi$, we have $M|L\mo \phi(a)\iff \exists b\in M_y,M|L\mo\psi(a,b)\iff$ (by induction hypothesis) $\exists b\in M_y,M\mo R_\psi(a,b)\iff M\mo \exists y R_\psi(a)\iff$ (by the $\exists$ clause) $M\mo R_\phi(a)$\\
- If $\phi=\forall y\psi$, we have $M|L\mo \phi(a)\iff \forall b\in M_y,M|L\mo \psi(a,b)\iff$ (by induction hypothesis) $\forall b\in M_y, M\mo R_\psi(a,b)\iff \not\exists b\in M_y, M\mo\neg R_\psi(a,b)\iff $ (by the $\neg$ clause) $\not\exists b\in M_y, M\mo R_{\neg\psi}(a,b)\iff M\not\mo \exists y R_{\neg\psi}(a,y)\iff$ (by definition of $R_{\neg\psi}(x,y)^M$) $M\not\mo \exists y \neg\psi(a,y)\iff M\mo \phi(a)$.\\
By induction on the complexity of $\phi$, we have $M|L\mo \phi(a)$ if and only if $M\mo R_\phi(a)$. Now if $\phi\in T$ is a sentence, we have $R_\phi\in T^G$, so $M\mo R_\phi$ and by what precedes we have $M|L\mo \phi$, so $M|L\mo T$.\\
(ii) If $f:M\to N$ is a $\lambda$-elementary embedding of models of $T$, it is a sorted map of $L^G$-structures between $M^G$ and $N^G$, because $L^G$ has the same sorts as $L$. Suppose that $M^G\mo \phi(a)$, where $\phi(x)$ is an atomic formula. If $\phi$ is an $L$-formula, we have $N\mo \phi(fa)$ by hypothesis on $f$, and if not, then $\phi(x)$ has the form $R(t_i(x):i<m)$, with the $t_i$'s $L$-terms and $R\in L^G-L$ : there exists an $L_{\lambda\omega}$-formula $\phi'(x)$ such that $R$ is equivalent to $R_{\phi'}$ modulo $T^G$, hence we have $M^G\mo \phi'(t_i(a):i<m)$, so $N^G\mo \phi'(t_i(fa):i<m)$, i.e. $N^G\mo \phi(fa)$, because $M^G,N^G\mo T^G$ by part (i) of the proposition. This shows that $G:(f:M\to N)\mapsto (f:M^G\to N^G)$ defines a functor from the category $\M(T)$ of models of $T$ with $\lambda$-elementary embeddings to the category $\M(T^G)$ of models of $T^G$ with $L^G$-homomorphisms.\\
In the other way round if $f:M\to N$ is an $L^G$-homomorphism between models of $T^G$, define $F(f:M\to
N)= (f:M|L\to N|L)$ : by the first part of the proposition, $M|L$ and $N|L$ are models of $T$, and if $M|L\mo \phi(a)$ with $\phi(x)\in L_{\lambda\omega}$, by the first part of the proof we have $M\mo R_{\phi}(a)$, so $N\mo R_{\phi}(fa)$ because $f$ is an $L^G$-homomorphism, and $N|L\mo \phi(fa)$ : $f$ is a $\lambda$-elementary $L$-embedding. This is now clear that $G$ and $F$ are inverse iosmorphisms of categories.
\end{proof}

It is well known that geometric (i.e. g-inductive) theories have classifying topoi, therefore the categories of set-theoretic models of g-inductive theories are essentially the categories of points of such topoi.
By geometric Morleyisation, the preceding theorem shows that the category of models of \emph{any} theory in $L_{\oo\omega}$ (with suitable morphisms, or working within a fixed fragment), is essentially the category of points of a such a topos. 

\begin{appendix}
\section{Existential Forcing}
In \cite{FAD}, Chapter 3, it is explained how the class of generic structures for Robinson's infinite forcing in an inductive class is a subclass of the class of existentially universal structures. Here we have not generalised the theory of infinite forcing in the more general positive context, but we have generalised in section \ref{EXINF} the notion of existentially universal model of \cite{FAD}, Chapter 1; this was natural in terms of an intepretation of infinitary logic in positive type spaces. In this appendix we interpret this as a forcing in models of $T_u$, for which the existential models are the generic ones. 

\begin{defi}
(i) Let $A\mo T_u$, $\phi(x)\in L_{\oo\omega}$ any formula and $a\in A_x$. We will say that \emph{$A$ (existentially) forces $\phi(a)$}, written $A\fo^e\phi(a)$, if for every positive type $p(x)\in S_x(T)$ such that $tp^+(a)\subset p$, one has $p\in [\phi(x)]$.\\
(ii) We will say that a model $M$ of $T_u$ is \emph{existentially generic}, if for every formula $\phi(x)\in L_{\oo\omega}$ and $a\in M_x$, one has
$$ M\mo\phi(a)\iff M\fo^e\phi(a).$$
\end{defi}

\begin{rem}
(i) Existential forcing is stable under continuations : if $A\mo T_u$ and $A\fo^e\phi(a)$, and if $f:A\to B\mo T_u$, then one has $B\fo^e\phi(fa)$.\\
(ii) If $M\mop T$, one has $M\fo^e\phi(a)$ if and only if $tp^+(a)\in [\phi(x)]$.\\
(iii) Contrary to classical finite and infinite forcing, existential forcing allows the consideration of infinitary formulas, thanks to the definition using the type spaces.
\end{rem}

\begin{lem}\label{PECTE}
If $M\mop T$, then for every formula $\phi(x)\in L_{\oo\omega}$ and $a\in M_x$, one has $M\fo^e\phi(a)$ or $M\fo^e\neg\phi(a)$.
\end{lem}
\begin{proof}
By (ii) of the preceding remark, if $M\not\fo^e\phi(a)$ we have $tp^+(a)\notin [\phi(x)]$, hence $tp^+(a)\in S_x(T)-[\phi(x)]=[\neg\phi(x)]$, %(CHECK)
i.e. $M\fo^e\neg\phi(a)$.
\end{proof}

This property is the usual notion of "genericity" for model theoretic forcing; however, genericity is construed for satisfaction to be equivalent to forcing, and because of the possible use of infinite conjunctions in building the formulas of $L_{\oo\omega}$ this will only be the case in existential models, as shows the following 

\begin{thm}
The existential models of $T_u$ are the existentially %nfinitely 
generic models.
\end{thm}
\begin{proof}
In the direct sense, let $M\moe T$. We proceed by induction on the complexity of $\phi(x)\in L_{\oo\omega}$ to showing that for $a\in M_x$ we have $M\mo \phi(a)\iff M\fo\phi(a)$; we let $p=tp^+(a)$. It suffices to treat the inductive steps of (infinite) conjunctions, negations and existential quantifications.\\
- If $\phi$ is atomic, then satisfaction and forcing trivially coincide\\
- If $\phi=\bigwedge \Phi$, we have $M\mo \phi(a)\iff$ for every $\psi\in \Phi$, $M\mo\psi(a)\iff$ (by induction hypothesis) $\forall\psi\in \Phi,\ M\fo\psi(a)\iff\ p\in\bigcap\{[\psi(x)]:\psi\in \Phi\}=[\bigwedge\Phi(x)]=[\phi(x)]$, i.e. $M\fo\phi(a)$\\
- If $\phi=\neg\psi$, then $M\mo\phi(a)\iff M\not\mo\psi(a) \iff$ (by induction hypothesis) $M\not\fo\psi(a)\iff$ (by Lemma \ref{PECTE}, as $M\mop T$) $M\fo\phi(a)$\\
- If $\phi=(\exists y)\psi$, suppose $M\mo \phi(a)$ : there exists $b\in M_y$ such that $M\mo \phi(a,b)$, hence by induction hypothesis we have $M\fo\psi(a,b)$, i.e. $q(x,y)=tp^+(a,b)\in [\psi(x,y)]$; if $\pi:S_{xy}(T)\to S_x(T)$ is the canonical projection, we have $p=\pi(q)$ and $\pi([\psi(x,y)])=[\phi(x)]$ (CHECK !!!), whence $p\in [\phi(x)]$ and $M\fo\phi(a)$. Conversely, if $M\fo\phi(a)$ we have $p\in [\phi(x)]$, hence there exists $q\in [\psi(x,y)]$ such that $p=\pi(q)$; the positive type $q(a,y)$ is consistent with $T$ and $D^+M$ (CHECK) hence as $M$ is existential there is a realisation $b$ of $q(a,b)$ in $M_y$, which means $tp^+(a,b)\in [\psi(x,y)]$, i.e. $M\fo\psi(a,b)$. By induction hypothesis we have $M\mo\psi(a,b)$ and by definition we get $M\mo\phi(a)$. The proof that $M$ is existentially generic is complete.\\
As for the reciprocal, let $M\mo T_u$ be existentially generic, and suppose $M$ is not an existential model of $T_u$ : there exists a partial positive type $\pi(x,b)$ with finite parameters in $M_y$, such that $\pi(x,b)$ is consistent with $D^+M\cup T_u$ but is not realised in $M$. This means we may find a homorphism $f:M\to N\mo T_u$ such that $N\mo(\exists x)\bigwedge\pi(x,fb)$, and we may suppose $N\moe T$, whereas $M\not\mo (\exists x)\bigwedge\pi(x,b)$. As $M$ is generic, we have $M\fo \neg(\exists x)\bigwedge\pi(x,b)$, and as forcing is stable under continuations we get $N\fo\neg(\exists x)\bigwedge\pi(x,fb)$; now by the first part of the proof $N$ is generic, hence $N\mo\neg(\exists x)\bigwedge\pi(x,fb)$, which is a contradiction. By \emph{reductio ad absurdum}, $M$ is an existential model of $T_u$.
\end{proof}

\end{appendix}

\end{document}